\documentclass[12pt]{amsart}
\usepackage{amstext,amsfonts,amssymb,amscd,amsbsy,amsmath,verbatim, mathrsfs, fullpage}
\usepackage[alphabetic,abbrev,lite]{amsrefs} 
\usepackage{ifthen}
\usepackage{color}
\usepackage{amsthm}
\usepackage{latexsym}
\usepackage[all]{xy}
\usepackage{enumerate}

\newtheorem{lemma}{Lemma}[section]

\newtheorem{prop}[lemma]{Proposition}
\newtheorem{cor}[lemma]{Corollary}

\newtheorem{claim*}{Claim}
\newtheorem{thm}[lemma]{Theorem}
\newtheorem{defn}[lemma]{Definition}
\newtheorem{example}[lemma]{Example}
\newtheorem{question}[lemma]{Question}

\theoremstyle{remark}
\newtheorem{remark}[lemma]{Remark}
\newtheorem{rem}[lemma]{Remark}

\newcommand{\PP}{\mathbb P}
\renewcommand{\P}{\PP}
\newcommand{\bA}{\mathbb A}
\newcommand{\A}{\bA}

\newcommand{\initial}{\operatorname{in}}
\newcommand{\Hilb}{\operatorname{Hilb}}
\newcommand{\Spec}{\operatorname{Spec}}

\newcommand{\Gr}{\operatorname{Gr}}

\newcommand{\cO}{{\mathcal O}}
\newcommand{\kk}{{\bf k}}

\newcommand{\FF}{\mathbb{F}}
\newcommand{\F}{\FF}
\newcommand{\GL}{{GL}}
\newcommand{\R}{\mathbb{R}}
\newcommand{\C}{\mathbb{C}}

\newcommand{\defi}[1]{\textsf{#1}} 

\newcommand{\beq}{\begin{displaymath}}
\newcommand{\eeq}{\end{displaymath}}
\newcommand{\bs}{\backslash}

\title{Furstenberg sets and Furstenberg schemes over finite fields}
\author{Jordan S. Ellenberg}
\address{Department of Mathematics, University of Wisconsin, Madison, WI 53706}
\email{ellenber@math.wisc.edu}
\urladdr{http://www.math.wisc.edu/~ellenber/}

\author{Daniel Erman}
\address{Department of Mathematics, University of Wisconsin, Madison, WI 53706}
\email{derman@math.wisc.edu}
\urladdr{http://www.math.wisc.edu/~derman/}
\thanks{ The first author is supported by NSF Grant DMS-1402620 and the second author was supported by NSF Grant DMS-1302057.}
\date{\today}

\begin{document}
\maketitle

\begin{abstract}
We give a lower bound for the size of a subset of $\F_q^n$ containing a rich $k$-plane in every direction, a {\em k-plane Furstenberg set}.
The chief novelty of our method is that we use arguments on non-reduced subschemes and flat families to derive combinatorial facts about incidences between points and $k$-planes in space.
\end{abstract}

\section{Introduction}

A central question in harmonic analysis is the {\em Kakeya conjecture}, which holds that a subset $S$ of $\R^n$ containing a unit line segment in every direction has Hausdorff dimension $n$.  Many refinements and generalizations of the Kakeya conjecture have appeared over the years.  For instance, one may loosen the condition on $S$, asking only that there be a line segment in every direction whose intersection with $S$ is large in the sense of Hausdorff dimension.

\begin{question}[Furstenberg set problem] Let $S$ be a compact subset of $\R^n$ such that, for every line $\ell \subset \R$, there is a line parallel to $\ell$ whose intersection with $S$ has Hausdorff dimension at least $c$.  What can be said about the Hausdorff dimension of $S$?
\end{question}

This problem was introduced by Wolff~\cite[Remark 1.5]{wolff99}, based on ideas of Furstenberg.  Wolff showed that $\dim S > \max(c+1/2, c n)$, and gave examples of $S$ with $\dim S = (3/2)c + (1/2)$.  

More generally, we can ask the same question about $k$-planes:

\begin{question}[k-plane Furstenberg set problem] Let $S$ be a compact subset of $\R^n$ such that, for every $k$-plane $W \subset \R^n$, there is a $k$-plane parallel to $W$ whose intersection with $S$ has Hausdorff dimension at least $c$.  What can be said about the Hausdorff dimension of $S$?
\end{question}

In this paper, we consider discrete and finite-field analogues of the $k$-plane Furstenberg set problem.

\begin{question}[k-plane Furstenberg set problem over finite fields] Let $\F_q$ be a finite field, and let $S$ be a subset of $\F_q^n$ such that, for every $k$-plane $W \subset \F_q^n$, there is a $k$-plane parallel to $W$ whose intersection with $S$ has cardinality at least $q^c$.  What can be said about $|S|$?
\label{q:kpFff}
\end{question}

We begin by recalling some known results about Question~\ref{q:kpFff} from the case $k=1$.  The method of Dvir's proof of the finite field Kakeya conjecture~\cite{dvir:kakeya} shows immediately that
\begin{equation}
\label{upper1high}
|S| \gtrsim q^{cn}
\end{equation}
and an elementary combinatorial argument shows that
\begin{equation}
\label{upper1low}
|S| \gtrsim q^{c + (n-1)/2}.
\end{equation}
(Here we write $|S|\gtrsim f(q,n,k,c)$ to mean that $|S| > Cf(q,n,k,c)$ for a constant $C$ which may depend on $n,k$  but which is independent of $q$.

In the other direction, Ruixiang Zhang has produced examples~\cite[Theorem 2.8]{zhang} showing that it is possible to have 
\beq
|S| \lesssim q^{(n+1)(c/2) + (n-1)/2}
\eeq
He conjectures that this upper bound is in fact sharp when $q$ is {\em prime}.  It is not sharp in general:  an example of Wolff~\cite[Remark 2.1]{wolff99} shows that when $q=p^2$ and $c = 1/2$ it is possible to have
\beq
|S| \lesssim q^{n/2}
\eeq
In particular, when $q=p^2$ both lower bounds \eqref{upper1high} and \eqref{upper1low} are sharp at the critical exponent $c=1/2$.

Much less is known about higher $k$.  In \cite[Conjecture 4.13]{eot}, the first author, with Oberlin and Tao, proposed a k-plane maximal operator estimate in finite fields.  When $k=1$, we prove the estimate~\cite[Theorem 2.1]{eot}, which bounds the Kakeya maximal operator and generalizes Dvir's theorem.  For general $k$ it remains a conjecture.  Its truth would imply that, for $S$ satisfying the hypothesis in Question~\ref{q:kpFff},
\beq
|S| \gtrsim q^{cn/k}.
\eeq
The main goal of the present paper is to show that this proposed lower bound for the $k$-plane Furstenberg problem is in fact correct.

\begin{prop} Let $S$ be a subset of $\F_q^n$.  Let $c \in [0,k]$.  Suppose that, for each $k$-plane $W \subset \F_q^n$, there is a $k$-plane $V$ parallel to $W$ with $|S\cap V| \geq q^c$.  Then
\beq
|S| > C q^{cn/k}
\eeq
for some constant $C$ depending only on $n$ and $k$.
\label{pr:kplane}
\end{prop}

We note that the condition $c \in [0,k]$ in the statement is superfluous, since a $k$-plane has at most $q^k$ points in all.  We include it in order to emphasize the analogy with Theorem\ref{th:main}.   See Remark~\ref{re:0k} for more discussion of this point..

In order to prove Proposition~\ref{pr:kplane}, we introduce an algebraic technique which is familiar in algebraic geometry but novel in the present context; that of {\em degeneration}.  

 A subset of $\F_q^n$ can be thought of as a reduced $0$-dimensional subscheme of the affine space $\A^n/\F_q$. Once this outlook has been adopted, it is natural to pose the Furstenberg set problem in a more general context, addressing all 0-dimensional subschemes, not only the reduced ones.
 
 Denote by $R$ the polynomial ring $\F_q[x_1, \dots , x_n]$. A $0$-dimensional subscheme $S$ of $\A^n$ is defined by an ideal $I \subset R$ such that $R/I$ is a finite-dimensional vector space over $\F_q$. The scheme $S$ is the affine scheme $\Spec R/I$. When $R/I$ is isomorphic to a direct sum of fields, $S$ is reduced and can be thought of as a set of points, and $I$ is the ideal of polynomials which vanish on the set of points.  
 
By contrast, a typical non-reduced example is the ``fat point''$S$ defined by $I = (x_1,\dots,x_n)^d$.\footnote{We think of $S$ as a copy of the origin which has been ``thickened'' infinitesimally; to evaluate a function at $S$ is to specify its values and all its partial derivatives of degree at most $d-1$.  In particular, to say a polynomial $f$ vanishes at $S$ is to say its partials of degree at most $d-1$ all vanish at the origin $0$, which is exactly to say it belongs to the ideal $I$.}  We denote $\dim_{\F_q} R/I$ by $|S|$; when $S$ is reduced (i.e. a set of points) then $|S|$ is the cardinality of the set of geometric points of $S$, just as the notation suggests. When $S$ is the fat point defined by $I = (x_1,\dots,x_n)^{d+1}$ then one can check that $|S|$ is $\binom{n}{d}$.

Our main theorem is that the lower bound on the size of a Furstenberg set asserted in Proposition~\ref{pr:kplane} applies word for word to Furstenberg schemes.

\begin{thm}\label{th:main}
Let $S$ be a $0$-dimensional subscheme of $\A^n/\F_q$.  Let $c \in [0,k]$.  Suppose that, for each $k$-plane $W \subset \A^n$ defined over $\F_q$, there is a $k$-plane $V$ parallel to $W$ with $|S \cap V| \geq q^c$.  Then 
\beq
|S| > C q^{cn/k}
\eeq
for some constant $C$ depending only on $n$ and $k$.
\end{thm}

\begin{rem}\label{re:0k}
The condition $c \in [0,k]$ is superfluous in Proposition~\ref{pr:kplane}, but not in Theorem~\ref{th:main}.  The subscheme of $\A^2$ cut out by the ideal $(x,y^N)$, for instance, intersects the line $x=0$ in degree $N$ and every other line in degree $1$.  Note that $N$ can be much larger than $q$; once we leave the world of reduced schemes, there is no a priori upper bound for the intersection of $S$ with a line!  In particular, the union of $q+1$ rotations of this scheme has $|S|$ on order $Nq$ and has $|S \cap V| \geq N$ for every $\F_q$-rational line $V \in \A^2$.  If $N = q^c$ and we allowed $c>1$, we would have $|S| \sim q^{c+1} \leq q^{2c}$, violating the theorem statement.
\end{rem}

Why is Theorem~\ref{th:main} easier to prove than its special case Proposition~\ref{pr:kplane}?  The answer involves certain parameter spaces for Furstenberg set problems (constructed in Section~\ref{sec:xm}) that allow us to vary the collection of points $S$.  The degenerate $0$-dimensional schemes form the boundary of this parameter space, and we can bound various functions for all $S$ by bounding them for these degenerate schemes.  Then as happens very often in algebraic geometry, after overcoming an initial resistance to degenerating to a non-smooth situation, we discover that the degenerate situation is actually easier than the original one.

Because our arguments are geometric in nature, they apply over general fields, not only finite fields.  The $k$-planes through the origin in $\A^n$ -- which we may think of as the set of possible directions -- is parametrized by the Grassmannian $\Gr(k,n)$.  Given $m,k,$ and $S$,  we let $\Sigma_{m,k}^S\subseteq \Gr(k,n)(\kk)$ denote the set of directions $\omega$ such that there is some $k$-plane $V$ in direction $\omega$ with $|S\cap V|\geq m$.  Our key technical idea is to observe that this set is more naturally thought of as the set of points on a scheme $X_{m,k}^S$, cut out by polynomial equations on the Grassmanian, and to closely study the properties of those defining equations.  This leads to the following more flexible theorem, from which Theorem~\ref{th:main} will follow without much trouble.

\begin{thm}\label{thm: new main}
Let $\kk$ be an arbitrary field and let $S$ be a $0$-dimensional subscheme of $\A^n/\kk$.  Let $X^S_{m,k}\subseteq \Gr(k,n)$ be the moduli space of directions of $m$-rich $k$-planes for $S$.    Then either:
\begin{enumerate}
	\item $X^S_{m,k}=\Gr(k,n)$ (that is, every $k$-plane direction is $m$-rich) and $|S|$ is at least $C_1 m^{n/k}$, for a constant $C_1$ depending only on $n$ and $k$; or

	\item  $\Sigma^S_{2m,k}$ is contained in a hypersurface $Z\subseteq \Gr(k,n)$ of degree at most $C_2\frac{|S|}{m}$, for a constant $C_2$ depending only on $n$ and $k$.
\end{enumerate}
\end{thm}

The connection between Theorem~\ref{thm: new main} and Theorem~\ref{th:main} involves a descending induction argument to reduce to the case $k=n-1$, combined with a simple observation about $\F_q$-points.  Working over $\F_q$, let $k=n-1$ and $m=q^c$ and assume that $|S| = o(q^{cn/(n-1)})$.  Then Theorem~\ref{thm: new main}(2) implies that $X^S_{m,n-1}$ lies in a hypersurface of degree $o(q)$.  However, this would contradict the hypotheses of Theorem~\ref{th:main}, as no hypersurface of degree less than $q$ can contain every $\F_q$-point of $\Gr(n-1,n)$.

\begin{rem} The bounds in Theorem~\ref{th:main} are sharp for every $c$; take $S$ to be the fat point of degree $q^c$ supported at the origin, so that the intersection of $S$ with any $k$-plane is on order $q^{ck}$ and $|S|$ is on order $q^{cn}$.  The bounds in Proposition~\ref{pr:kplane}, however, are {\em not} sharp, or at least are not sharp over the whole range $c \in [0,k]$.  Already when $k=1$, we see that the bound $|S| \gtrsim q^{cn}$ fails to be sharp only when $c < 1/2$ (and when $q$ is prime, it fails to be sharp for $c<1$, by a result of Zhang~\cite[Theorem 1.4]{zhang}.)  The results of the present paper suggest that purely algebraic arguments apply to $0$-dimensional schemes over arbitrary fields and are effective at controlling $k$-planes which are very rich in incidences, while more combinatorial arguments, which apply only to point sets, may be stronger tools for bounding incidences arising from $k$-planes which are not so rich in points. 
\end{rem}

\begin{rem}  The scheme-theoretic methods of this paper may seem very distant from anything that could be of use in Euclidean problems.  But there is an interesting similarity between the degeneration method used here and the method used by Bennett, Carbery, and Tao in their work on the multilinear Kakeya conjecture~\cite{bct}.  Their work required bounding an $\ell^p$ norm on a sum of characteristic functions of thin tubes in different directions; one idea in their paper involves sliding all these tubes towards $0$ until they all intersect at the origin, and showing that the quantities they are trying to bound only go up under that process.  (See especially \cite[Question 1.14]{bct}.)  An argument of this kind can also be found in \cite{bbc}. Our method is in some sense very similar; the main degeneration we consider is a dilation, where all points in $S$ move to $0$ and all lines in direction $\omega$ slide to the line through $0$ in direction $\omega$.   Our hope is that the large existing body of work in this area of algebraic geometry may provide more ideas for carrying out ``degeneration" arguments in the Euclidean setting.
\end{rem}

This paper is organized as follows.  In Section~\ref{sec: notation} we outline some notation that we will use throughout the paper.  In Section~\ref{sec:sketch of proof} we give a detailed sketch of the proof of Theorem~\ref{th:main}.  Section~\ref{sec:xm} contains much of the technical work of the paper, as we construct the schemes $X^S_{m,k}$ and study some of their essential properties.  In Section~\ref{sec:equality}, we focus on the special case of when $X^S_{m,k}$ equals the entire Grassmanian, as this plays a central role in our main results.  Sections~\ref{sec:proof of main theorem} and ~\ref{sec: proof of furstenberg} then contain the proofs of Theorems~\ref{th:main} and \ref{thm: new main}.  In Section~\ref{sec:restriction} we discuss an approach to the $k$-plane restriction conjecture of \cite{eot}, and Section~\ref{sec:examples} concludes with a few examples.

\subsection*{Acknowledgments}
We thank Terry Tao for introducing us to the question and for noting the similarity between our method and that of \cite{bct}, and Ruixiang Zhang for helpful discussions about the Furstenberg set problem. 

\section{Notation and Background}\label{sec: notation}
In this subsection we gather some of the notation that we will use throughout.  For reference, we also gather some of the notation from the introduction. Throughout, $\kk$ will denote an arbitrary field and $\F_q$ will denote a finite field of cardinality $q$.  If $Z$ is a scheme over $\kk$ and $\kk'$ is a field over $\kk$, then we write $Z(\kk')$ for the $\kk'$-valued points of $Z$.

We use $S$ to denote a $0$-dimensional subscheme of $\A^n/{\kk}$, and $I_S$ to denote its defining ideal, so that  $S = \Spec \kk[x_1, \ldots, x_n] / I_S$.  We set $|S|:=\dim_{\kk} \kk[x_1, \ldots, x_n] / I_S$.  If $S$ is a $0$-dimensional subscheme of $\A^n/\kk$ as above, and $V$ is a linear space cut out by linear forms $\ell_1, \ldots \ell_s$, we mean by $S\cap V$ the scheme-theoretic intersection $\Spec \kk[x_1, \ldots, x_n]/(I_S + (\ell_1, \ldots, \ell_s))$.  We say that $V$ is \defi{$m$-rich for $S$} if $|S\cap V|\geq m$.

We also review a few concepts about ideals.  Let $R$ be the coordinate ring of an algebraic variety $Z\subseteq \PP^r$ and let $J\subseteq R$ be an ideal.  The radical of $J$, denoted $\sqrt{J}$, is the ideal
\[
\sqrt{J} = \{ f\in R | f^n\in J \text{ for some } n\geq 0\}.
\]
The ideal $\sqrt{J}$ contains all functions that vanish on the subset $V(J)\subseteq Z$.  The $m$'th power of $J$, denoted $J^m$, is the ideal
\[
J^m =\{ f=f_1f_2\cdots f_m | f_i\in J\} \subseteq R.
\]
If $J$ is prime, then we can define the $m$th symbolic power of $J$, denoted $J^{(m)}$, to be the $J$-primary component of  $J^m$.  A similar definition is used for more general ideals $J$.

However, if $Z$ is a smooth variety, then there is a more geometric characterization of symbolic powers due to Zariski and Nagata~\cite[\S3.9]{eisenbud-book}.  Assume that $I\subseteq R$ is radical.  Then the symbolic power $I^{(m)}$ equals the ideal of functions that vanish with multiplicity $m$ along the locus $V(I)\subseteq X$.  In particular, if we write $\mathfrak m_x$ for the homogeneous prime ideal in $R$ corresponding to a point $x\in V(I)$, then
\[
I^{(m)} = \cap_{x\in V(I)} \mathfrak m_x^m.
\]
In general, we have $I^m\subseteq I^{(m)}$, but not equality.

\section{Sketch of the proof}\label{sec:sketch of proof}
We begin with an overall sketch of the proof of Theorem~\ref{th:main}.  The idea is as follows.  Let $S$ be a $0$-dimensional subscheme of $\A^n$.  Then we can degenerate $S$ by {\em dilation} to a subscheme $S_0$ of $\A^n$ which is supported at the origin, and which has $|S_0| = |S|$.  We may think of $S_0$ as the limit of $tS$ as $t$ goes to $0$.  If $V$ is a $k$-plane with $|S\cap V| \geq q^c$, then $|S_0 \cap V_0| \geq q^c$ where $V_0$ is the $k$-plane through the origin parallel to $V$.  In particular, the Furstenberg condition on $S$ implies that $|S_0 \cap V_0| \geq q^c$ for {\em every} $\F_q$-rational $k$-plane through the origin in $\A^n$.  The  supremum over a parallel family of $k$-planes has disappeared from the condition, which allows for an easy induction argument reducing us to the case $k=n-1$.  Namely:  given that Theorem~\ref{th:main} holds for $k=n-1$, let $W_0$ be a $(k+1)$-plane through the origin in $\A^n$.  Every $k$-plane $V_0$ through the origin in $\A^n$ satisfies $|S_0 \cap V_0| \geq q^c$, so Theorem~\ref{th:main} tells us that $|S_0 \cap W_0| \geq q^{c(k+1)/k}$ for every choice of $W_0$.  Iterating this argument $n-k$ times gives us the desired bound $|S_0| \geq q^{cn/k}$.

This leaves the proof of Theorem~\ref{th:main} in the hyperplane case.  We prove this proposition by considering a geometric version of the Radon transform.  The Radon transform may be thought of as a function $f_S$ on the Grassmannian $\Gr(n-1,n) \cong \P^{n-1}$, defined by
\beq
f_S(V_0) = |S_0 \cap V_0|.
\eeq
 (Usually the Radon transform is thought of as a function on all hyperplanes, not only those through the origin; in this case, since $S_0$ is supported at the origin, the Radon transform vanishes on those hyperplanes not passing through the origin.) 
 
Unfortunately, the notion of real-valued function doesn't transfer to the scheme-theoretic setting very neatly; what works better is the notion of level set.  Naively:  we might define
\beq
X_{m,n-1}^{S_0} = \{V_0 \in \Gr(n-1,n): |V_0 \cap S_0|  \geq m\}
\eeq
as the set of $m$-rich hyperplanes through the origin.

It turns out, however, that to make the notion of Radon transform behave well under degeneration, we need to think of the level set $X_{m,n-1}^{S_0}$ not as a subset of the $\kk$-points of $\Gr(n-1,n)$, but as a {\em subscheme} of $\Gr(n-1,n)$. In fact, for easy formal reasons, it is a closed subscheme.  This viewpoint has the further advantage that we can argue geometrically, without any reference to the field over which we are working.  We explain the definition of $X_{m,n-1}^{S_0}$ and its behavior under degeneration of $S$ in Section~\ref{sec:xm}, which is where most of the technical algebraic geometry is to be found.

We show in Proposition~\ref{prop:Xm equals Gr} that for $m \gg N^{(n-1)/n}$, the level scheme $X_{m,n-1}^{S_0}$ is not the whole of $\Gr(n-1,n)$.  This argument involves a further degeneration, a {\em Gr\"{o}bner degeneration} from $S_0$ to a member of a yet more restricted class of schemes called Borel-invariant subschemes. Thus, $X_{m,n-1}^{S_0}$, being Zariski closed, is contained in a proper hypersurface.  We bound the degree of this hypersurface in part (2) of Theorem~\ref{thm: new main} (by means of explicit defining equations), and this provides the final piece of the proof of Theorem~\ref{th:main}.

\section{The schemes $X_{m,k}^S$}
\label{sec:xm}

Beginning in this section, we work over an arbitrary field $\kk$ and omit the field $\kk$ from most of the notation, e.g. writing $\A^n$ in place of $\A^n/\kk$.  It may be useful for the reader to imagine that $\kk=\F_q$.   

Initially, we let $S$ be a collection of points (a reduced $0$-dimensional scheme) in $\A^n$.  Let $S_0$ be the degeneration of $S$ by the {\em dilation} action.  This can be defined concretely as follows.  Let $I \subset \kk[x_1, \ldots, x_n]$ be the ideal of polynomials vanishing at $S$.  If $t$ is an element of $\kk^*$, then the ideal of functions vanishing at the dilation $S_t:=tS$ is precisely
\beq
I_t = \{f(t^{-1} x_1, \ldots, t^{-1} x_n): f \in I\}.
\eeq
We then ask what happens as ``$t$ goes to $0$."  Of course, this doesn't literally make sense since $\kk$ is not necessarily $\R$ or $\C$, but may be a finite field or something even more exotic.  Nonetheless, if one thinks of $t$ as getting ``smaller", than $f(t^{-1} x_1, \ldots, t^{-1} x_n)$ will be ``dominated" by its highest-degree term $f_d$, a homogeneous polynomial.  So the dilation $I_0$ is defined to be the homogeneous ideal generated by the highest-degree terms of polynomials in $I$, and $S_0 = \Spec \kk[x_1,\dots,x_n] / I_0$ is the subscheme of $\A^n$ cut out by the vanishing of the polynomials of $I_0$.  It's clear that $|S_t| = |S|$ for all $t \in \kk^*$; in fact, $S_t$ is isomorphic to $S$.  It turns out that $S_0$, while not typically isomorphic to $S$, does satisfy $|S_0| = |S|$, as a consequence of the Hilbert polynomial being constant in flat families.

Now let $\Sigma^S_{m,k}\subseteq \Gr(k,n)(\kk)$ denote the set of directions of all $k$-planes that are $m$-rich for $S$.  As observed in the first paragraph of Section~\ref{sec:sketch of proof}, if $S_0$ is the degeneration of $S$ by the dilation action, then $\Sigma^{S_0}_{m,k}$ will contain $\Sigma^S_{m,k}$.  This follows from the following standard lemma.

\begin{lemma}  Let $V$ be a $k$-plane in $\A^n$ such that $|S\cap V| \geq m$.  Let $V_0$ be the $k$-plane through the origin parallel to $V$.  Then $|S_0 \cap V_0| \geq m$.
\label{capdilate}
\end{lemma}

Geometrically, we think of the rationale for Lemma~\ref{capdilate} as follows:  if $V$ is a $k$-plane with $|S\cap V| \geq m$, then for every $t$, the dilation $tS$ is contained in the plane $tV$.  As $t$ goes to $0$, $tV$ converges to the $k$-plane $V_0$ parallel to $V$ and through the origin, and we find that $|S_0\cap V_0| \geq m$.

\begin{proof}
We consider $S_t\cap V_t$ as a family of $0$-dimensional schemes over $\A^1 = \Spec(\kk[t])$.  When $t\ne 0$, the degree of the fiber is constant and equals $|S\cap V|$. By semicontinuity (see~\cite[Theorem III.12.8]{hartshorne} or Proposition~\ref{prop:hilb poly semicontinuity} below) we have $|S_0\cap V_0|\geq |S\cap V|$.
\end{proof}

We henceforth focus on the case on the case where $S$ is a non-reduced $0$-dimensional scheme supported at the origin and defined by a homogeneous ideal $I_S\subseteq \kk[x_1,\dots,x_n]$.  We let $N:=|S|=\dim_{\kk} \kk[x_1,\dots,x_n]/I_S$.  

\subsection{Constructing the schemes $X_{m,k}^S$}

In this section, and henceforth, we adopt a more geometric point of view, replacing the set $\Sigma^S_{m,k}$ with a moduli scheme $X_{m,k}^S$ of $m$-rich $k$-plane directions, satisfying $X_{m,k}^S(\kk) = \Sigma^S_{m,k}$.

We let $H^N$ stand for the $\mathbb G_m$-equivariant Hilbert scheme $\Hilb^N(\mathbb A^n)$.  This scheme parametrizes homogeneous ideals $J\subseteq \kk[x_1,\dots,x_n]$ such that $\dim_{\kk} \kk[x_1,\dots,x_n]/J =N$; equivalently, it parametrizes zero-dimensional subschemes of $\mathbb A^n$ that are equivariant with respect to the $\mathbb G_m$ dilation action.  (Note that $H^N$ decomposes as a union of multigraded Hilbert schemes depending on the Hilbert function of $J$.  See \cite[Theorem~1.1]{haiman-sturmfels} for details.)

We want to define an incidence scheme that parametrizes pairs $(V, S)$ where $V$ is an $m$-rich $k$-plane for $S$.
We will write $[S]\in H^N$ for the point corresponding to $S$ and we will similarly write $[V]\in \Gr(k,n)$ for the class corresponding to a $k$-plane $V$.
We define our incidence scheme as follows.  Let $\mathcal I_H\subseteq \cO_{H^N}[x_1,\dots,x_n]$ be the ideal sheaf for the universal family over the Hilbert scheme.  We write $\cO_U:=\cO_{H^N}[x_1,\dots,x_n]/\mathcal I$ for the structure sheaf of the universal family over $H^N$.  Note that $\cO_U$ (or more precisely its pushforward, though we ignore this in the notation) is a vector bundle of rank $N$ on $H^N$.  

Now let $W$ be the vector space $\langle x_1,\dots,x_n\rangle$.  There is a tautological sequence
\[
0\to \mathcal S \to \cO_{\Gr}\otimes W \to \mathcal Q \to 0
\]
of vector bundles on $\Gr(k,n)$ of rank $n-k, n,$ and $k$ respectively.  Note that $\cO_{\Gr}\otimes W$ is the space of linear forms in the algebra $\cO_{\Gr}\boxtimes \kk[x_1,\dots,x_n]$ and the fiber of $\mathcal S$ over the point $[V]\in \Gr(k,n)$ is the $(n-k)$-dimensional space of linear forms vanishing at $V$.  In other words, we can rewrite this as a map $\mathcal S=\mathcal S\boxtimes \kk \to \cO_{\Gr}\boxtimes \kk [x_1,\dots,x_n]$.

Tensoring the righthand factors by $-\otimes_{\kk[x_1,\dots,x_n]} \cO_U$ then yields
a map of of vector bundles on $\Gr(k,n)\times H^N$
\begin{equation}\label{eqn:Phi}
\Phi\colon \mathcal S \boxtimes \cO_U\to \mathcal \cO_{\Gr}\boxtimes \cO_U
\end{equation}
of ranks $(n-k)N$ and $N$ respectively.  We define $Y_{m,k}\subseteq \Gr(k,n)\times H^N$ by the vanishing of the $(N-m+1)\times (N-m+1)$ minors of $\Phi$.

We claim that the points of $Y_{m,k}$ in $\Gr(k,n) \times H^N$ are precisely those pairs $([V],[S])$ such that $|S\cap V| \geq m$.  To see this, we consider a fixed $0$-dimensional scheme $S$ such that $|S|=N$.  
\begin{defn}\label{defn:XmkS}
Fix $m,k$ and $S$ as above, with $|S|=N$.  We define $X_{m,k}^S$ to be the fiber of $Y_{m,k}$ over $[S]\in H^N$:
\[
\xymatrix{
X_{m,k}^S\ar[r]\ar[d]&Y_{m,k}\ar[d]\\
[S]\ar[r]&H^N
}
\]
\end{defn}
The defining equations of $X_{m,k}^S$ are given by the $(N-m+1)\times (N-m+1)$ minors of
\begin{equation}\label{eqn:Phibar}
\overline{\Phi}\colon \mathcal S\boxtimes \cO_S \to \cO_{\Gr}\boxtimes \cO_S,
\end{equation}
which is a map of vector bundles on $\Gr(k,n)$.
At a point $[V] \in \Gr(k,n)$ the cokernel of $\overline{\Phi}$ defines the structure sheaf of $S\cap \Lambda$.  Thus, $[V]\in X_{m,k}^S$ if and only if the the cokernel has degree at least $m$, which is exactly what we wanted.  In particular, as a set,
\[
X_{m,k}^S(\kk) = \{[V]  \text{ where $V$ is $m$-rich for $S$}\}\subseteq \Gr(k,n)(\kk) = \Sigma_{m,k}^S
\]
which is what we claimed above.

\subsection{Local structure}
Fix $k$ and $S$ as above.   We embed $\Gr(k,n)\subseteq \PP^{\binom{n}{k}-1}$ via the Pl\"ucker embedding so that $X^S_{m,k}\subseteq \PP^{\binom{n}{k}-1}$ for all $m$.

\begin{defn}
We let $J_{X_m}$ be the ideal of $(N -m+1)\times (N -m+1)$ minors of $\overline{\Phi}$ considered as an ideal on the homogeneous coordinate ring of $\Gr(k,n)$, and we let $I_{X_m}:=\sqrt{J_{X_m}}$ denote the radical of $J_{X_m}$.
\end{defn}
Note that $J_{X_m}\subseteq I_{X_m}$ and that both ideals define the same closed subscheme, but they may not be equal.  In particular, it is possible that there could be low degree polynomials vanishing on $X_{m,k}^S$ (and hence lying in $I_{X_m}$) which do not come from $J_{X_m}$.

\begin{lemma}\label{lem:degree defining equations}
There is a constant $C$ depending only on $n$ and $k$ such that the ideals $J_{X_m}$ is generated in degree at most $C(|S|-m+1)$.  It follows that $I_{X_m}$ contains an element of degree at most $C(|S|-m+1)$, i.e. that $X_m$ lies on a hypersurface of degree at most $C(|S|-m+1)$.
\end{lemma}
\begin{proof}
Let $N:=|S|$, so that we we can identify $\cO_S$ with $\kk^N$, and have $\overline{\Phi}: \mathcal S^{\oplus N} \to \cO_{\Gr}^{\oplus N}$.  Let $\cO_{\Gr}(1)$ be the Pl\"ucker line bundle on $\Gr(k,n)$.  There is some constant $d$, depending only on $k$ and $n$, such that $\mathcal S\otimes \cO_{\Gr}(d)$ is globally generated.  If $M:=\dim H^0(\Gr(k,n),\mathcal S\otimes \cO_{\Gr}(d))$, we have a surjection
\[
\cO_{\Gr}(-d)^{\oplus M\cdot N}\to \mathcal S^{\oplus N}.
\]

We now take $k\times k$ minors of $\Phi$, with $k=|S|-m+1$, which yields the ideal sheaf $\mathcal J_{X_m}$ corresponding to the ideal $J_{X_m}$ as the image of the map
\[
\bigwedge^k \Phi = \bigwedge^k \mathcal S^{\oplus N} \otimes \bigwedge^k \cO_{\Gr}^{\oplus N} \to \cO_{\Gr}.
\]
There is a natural surjection 
\[
\bigwedge^k \cO_{\Gr}(-d)^{\oplus M\cdot N}\otimes \bigwedge^k (\cO_{\Gr}^{\oplus N})^*
\to \bigwedge^k \mathcal S^{\oplus N} \otimes \bigwedge^k (\cO_{\Gr}^{\oplus N})^*
\]
which in turn surjects onto $\mathcal J_{X_m}$.  This proves that $J_{X_m}$ is generated in degree at most $d\cdot k$.  Since $I_{X_m}\supseteq J_{X_m}$ the second statement follows immediately.

\end{proof}

\begin{lemma}\label{lem:local Xm structure}
Assume that the $k$-plane $V$ satisfies $|S\cap V|\geq m$.  Let $\mathfrak m_V$ be the maximal ideal of the point $[V]\in \Gr(k,n)$.   If $m\geq \ell$ then
\[
J_{X_\ell}\subseteq \mathfrak m_V^{m-\ell+1}.
\]
\end{lemma}
\begin{proof}
We localize the map $\overline{\Phi}$ from \eqref{eqn:Phibar} at the point $[V]$ to get an $N(n-k)\times N$ map of free $\cO_{\Gr,[V]}$-modules.  After choosing bases, we can write this as a matrix, and we denote this by $\overline{\Phi}_{[V]}$.  Since $V$ intersects $S$ in degree $m$, it follows that $\overline{\Phi}_{[V]}$ has rank $N-m$.  We are over a local ring, so every entry of this matrix is either a unit or lies in the maximal ideal $\mathfrak m_{V}$.  The matrix thus has a minor of size $(N-m)\times (N-m)$ that is a unit, and so after inverting this element and performing row and column operations, we can rewrite
\[
\overline{\Phi}_{V}=\begin{pmatrix} Id_{N-m} & 0\\ 0&A\end{pmatrix}
\]
where $A$ is an $N(n-k)-(N-m)\times m$ matrix consisting entirely of entries lying in the maximal ideal (otherwise $\Phi_{v}$ would have rank $N-m+1$).

It follows that the ideal of $(N-\ell+1)\times (N-\ell+1)$-minors of $\overline{\Phi}_{V}$ is the same as the ideal of $(m-\ell+1)\times  (m-\ell+1)$ minors of $A$, and every such minor is a determinant of entries lying in $\mathfrak m_{V}$, and this yields the desired inclusion.
\end{proof}
\begin{cor}\label{cor:local Xm symbolic}
If $m\geq \ell$ then $J_{X_\ell}$ belongs to the symbolic power $I_{X_m}^{(m-\ell+1)}$.
\end{cor}
\begin{proof}
Since the Grassmanian is smooth, this follows from Lemma~\ref{lem:local Xm structure} and the Zariski-Nagata Theorem.  See also the discussion in Section~\ref{sec: notation}.
\end{proof}

\subsection{Semicontinuity}

The total parameter space $Y_{m,k}$ enables us to study properties of $X_{m,k}^S$ as $S$ varies in $H^N$.  We can assign $X_{m,k}^S$ a Hilbert polynomial in $\mathbb Q[t]$ via the Pl\"ucker embedding of the Grassmanian.  We compare polynomials in $\mathbb Q[t]$ by saying that $f(t)>g(t)$ if this is true for all $t\gg 0$.
\begin{prop}\label{prop:hilb poly semicontinuity}
Let $Z\subseteq \PP^r\times V$ be a closed subscheme and let $\pi: Z\to V$ the projection map.  For $v\in V$ we defined $Z_v$ as the scheme-theoretic fiber of $\pi$ over $v$.  The Hilbert polynomial of the fibers of $\pi$ are upper semicontinuous in the following sense: fix any $f(t)\in \mathbb Q[t]$; the set
$
\{ v\in V | \text{ the Hilbert polynomial of } Z_v \text{ is at least } f(t)\}
$
is a closed subset of $V$.
\end{prop}
\begin{proof}
This is a standard fact but we include a short proof here for completeness.

Fix some Hilbert polynomial $p(t)$ on $\PP^r$.  The Gotzmann number provides a bound $t_p$ such that, for any projective subscheme $Z'\subseteq \PP^r$ with Hilbert polynomial $p(t)$, the Hilbert function and Hilbert polynomial of $Z$ in all degrees $\geq t_p$ (see e.g.~\cite[Chapter~4.3]{bruns-herzog-book}).

We may choose a flattening stratification for $\pi$, i.e. we may write $V$ as a finite disjoint union $V=\sqcup_{i=1}^s V_i$ such that the induced maps from $Z\times_{\PP^r\times V} V_i \to V_i$ are all flat.  Since the Hilbert polynomial is constant in a flat family~\cite[Theorem III.9.5]{hartshorne}, we see that only $s$ distinct Hilbert polynomials appear among the fibers of $\pi$.  We set $t_0$ to be the maximum of all of the Gotzmann numbers of these Hilbert polynomials.  Then for all $t\geq t_0$ and for all $[S]\in H^N$, the Hilbert polynomial of $X_{m,k}^S$ equals the Hilbert function in degrees $t\geq t_0$. 

We next observe that insisting that the Hilbert function be at least a certain value is a closed condition.  Hence for any $f(t)\in \mathbb Q[t]$, the set of fibers whose Hilbert polynomial is at least $f(t)$ is an intersection of closed subschemes, and is thus a closed subscheme.
\end{proof}

In the present paper, we use Proposition~\ref{prop:hilb poly semicontinuity} only through its easy corollary below.  We include Proposition~\ref{prop:hilb poly semicontinuity} because we believe the more general formulation may be useful in later applications of the techniques introduced in this paper.

\begin{cor}
\label{co:grclosed}
Let $\mathbf S\subseteq \mathbb A^n\times \mathbb A^1$ be a flat family of $0$-dimensional schemes over $\mathbb A^1$.  Write $S_t$ for the fiber of $\mathbf S$ over $t\in \mathbb A^1$.  If $X^{S_t}_{m,k}=\Gr(k,n)$ for all $t\ne 0$, then $X^{S_0}_{m,k}=\Gr(k,n)$.
\end{cor}

\begin{proof} $\Gr(k,n)$ has maximal Hilbert polynomial among closed subschemes of $\Gr(k,n)$, so the subscheme $W$ of $\A^1$ parametrizing those $t$ such that $X^{S_t}_{m,k}=\Gr(k,n)$ is closed, by Proposition~\ref{prop:hilb poly semicontinuity}, but $W$ is dense by hypothesis, so $W$ is all of $\A^1$.
\end{proof}

\section{Criteria for $X_{m,k}^S=\Gr(k,n)$}\label{sec:equality}
One boundary case that will feature prominently in the proofs of both Theorem~\ref{th:main} and Theorem~\ref{thm: new main} is the case where $X_{m,k}^S=\Gr(k,n)$ as schemes, or equivalently when all $k$-planes (even those defined over field extensions of $\kk$) are $m$-rich for $S$.  This is impossible for a reduced $0$-dimensional scheme, but it can happen when $S$ is non-reduced.  

For instance, if $S$ is the fat point defined by $(x_1,\dots,x_n)^{d+1}$ then every $k$-plane will be $m:=\binom{d+k}{k}$-rich.  Observe that, in this case, $|S|=\binom{d+n}{n}\approx m^{n/k}$.  This suggests the following result, which gives a similar lower bound on $|S|$ whenever $X_{m,k}^S=\Gr(k,n)$.

\begin{prop}\label{prop:Xm equals Gr}
Suppose that $X_{m,k}^S = \Gr(k,n)$. Then there is a constant $C$ depending only on $n$ and $k$ such that $|S|\geq Cm^{n/k}$.  More precisely, if $m\geq \binom{b}{k}$ then $|S|\geq \binom{b+(n-k)}{n}$.

\end{prop}

Our proof of Proposition~\ref{prop:Xm equals Gr} relies on a further degeneration to a Borel fixed scheme, and this is most easily defined over an infinite field.  Note, that the hypotheses and conclusions of the above proposition are unchanged under field extension, and so we may prove this proposition after extending the field $\kk$.  Over a field $\kk$, we let $B\subseteq \GL_n(\kk)$ be the Borel subgroup consisting of invertible upper triangular matrices, and we let $B$ act on $\kk[x_1,\dots,x_n]$ in the natural way.  When $\kk$ is infinite, then we say that a subscheme $Z\subseteq \mathbb A^n$ is \defi{Borel fixed} if $Z$ is invariant under the action of $B$.

\begin{remark}\label{rmk:borel degeneration}
Under the assumption that $\kk$ is infinite, we can degenerate any subscheme to a Borel fixed subscheme via the following recipe.
Consider a subscheme $Z$ defined by the ideal $J\subseteq \kk[x_1,\dots,x_n]$.  Fix a term order $\preceq$ satisfying $x_1\preceq x_{2}\preceq \cdots \preceq x_n$.  We choose a general element of $B$ (this is where we use the assumption the $\kk$ is infinite), apply that element to $J$, and then take the initial ideal with respect to $\preceq$ to obtain a new ideal $J'$.    The subscheme $Z'\subseteq \mathbb A^n$ defined by $Z'$ will be Borel fixed~\cite[Theorem 15.20]{eisenbud-book}.

In addition, the monomials not in $J'$ will be closed under the operation (called a {\em Borel move}) of replacing $x_j$ with $x_i$ for $i < j$.  We thus define a {\em Borel-fixed set of monomials} as a collection of monomials satisfying this property, and where the complementary set of monomials is closed under multiplication by each $x_i$.  See~\cite[Section~15.9]{eisenbud-book} for an introduction to Borel-fixed ideals.
\end{remark}

\begin{lemma}\label{lem:frontier}
Let $\Lambda$ be a Borel-fixed set of monomials in $x_1, \ldots, x_n$ such that $|\Lambda| \geq {a \choose n}$, and let $\Lambda_0$ be the subset of $\Lambda$ in which the power of $x_1$ is $0$.  Then
\beq
|\Lambda| - |\Lambda_0| \geq {a -1 \choose n}.
\eeq
\end{lemma}

\begin{proof}
We argue by induction on $n$.  For $n=1$ the assertion is clear; $|\Lambda| - |\Lambda_0| = |\Lambda| -1 \geq a-1$.

Now we suppose the lemma holds in $n-1$ variables.  We denote by $\Lambda_k$ the set of monomials $m$ in $x_2, \ldots, x_n$ such that $x_1^k m$ lies in $\Lambda$.  (In particular, the definition of $\Lambda_0$ conforms with our existing notation.)  We note that $\Lambda_k$ is a Borel-fixed set of monomials, so we can apply our inductive hypothesis.  Plainly, $\Lambda_{k+1} \subset \Lambda_k$.   

Let $m$ be a monomial in $\Lambda_k$ and suppose $mx_i$ lies in $\Lambda_k$ for some $i \in 2,\ldots, n$.  Then $x_1^{k+1}m$ must also lie in $\Lambda$, since it differs from $x_1^k m x_i \in \Lambda$ by a Borel move.  In particular, $m$ lies in $\Lambda_{k+1}$.  Thus, any element in $\Lambda_k \bs \Lambda_{k+1}$ must lie on the {\em frontier} of $\Lambda_k$; that is, $mx_i$ is not in $\Lambda_k$ for any $i \in 2,\ldots,n$.

Suppose $|\Lambda_0| \geq {b \choose n-1}$.  Let $\Lambda_{00}$ be the set of monomials in $\Lambda_0$ in which the power of $x_2$ is $0$.  No two elements on the frontier of $\Lambda_0$ can differ by a power of $x_2$; it follows that the cardinality of the frontier is at most $|\Lambda_{00}|$.  Combining this with the argument in the previous paragraph, we have
\beq
|\Lambda_1| = |\Lambda_0| - |\Lambda_0 \bs \Lambda_1| \geq |\Lambda_0| - |\Lambda_{00}| \geq {b-1 \choose n-1}
\eeq
where the latter inequality follows by applying the inductive hypothesis to $\Lambda_0$.  Proceeding by induction, we have that $|\Lambda_k| \geq {b-k \choose n-1}$.  Finally,
\begin{equation}
|\Lambda| - |\Lambda_0| = \sum_{k=1}^\infty |\Lambda_k| \geq \sum_{k=1}^\infty {b-k \choose n-1} = {b \choose n}.
\label{ll0}
\end{equation}

We can now prove the theorem.  We have that $|\Lambda| \geq {a \choose n}$.  If $|\Lambda_0| \leq {a-1 \choose n-1}$, then
\beq
|\Lambda| - |\Lambda_0| \geq {a \choose n} - {a-1 \choose n-1} = {a-1 \choose n}
\eeq
and we are done.
On the other hand, if $|\Lambda_0| \geq {a-1 \choose n-1}$, then \eqref{ll0} yields
\beq
|\Lambda| - |\Lambda_0| \geq {a-1 \choose n}
\eeq
So the desired conclusion holds in either case.

\end{proof}

\begin{cor}
\label{co:frontier}
Let $\Lambda$ be a Borel-fixed set of monomials in $x_1, \ldots, x_n$, and let $\Lambda_0$ be the subset of $\Lambda$ in which the power of $x_1$ is $0$, and suppose $|\Lambda_0| \geq {b \choose n-1}$.  Then
\beq
|\Lambda| - |\Lambda_0| \geq {b \choose n}.
\eeq
\end{cor}

\begin{proof}
Immediate from \eqref{ll0} and the paragraph preceding it.
\end{proof}

\begin{proof}[Proof of Proposition~\ref{prop:Xm equals Gr}]
Without loss of generality, we may assume that $\kk$ is an algebraically closed field.

We first prove the statement in the special case $k=n-1$.  Suppose $X_{m,n-1}^S = \Gr(n-1,n)$.  Let $S_{\initial}$ be the $0$-dimensional subscheme defined by a Borel-fixed degeneration of the defining ideal of $S$, as in Remark~\ref{rmk:borel degeneration}.   By \cite[Theorem~15.17]{eisenbud-book}, there is a flat family over $\A^1$ where the fiber over $0\in \A^1$ is $S_{\initial}$ and every other fiber is isomorphic to $S$ via an isomorphism that extends to a linear automorphism of $\A^n$.  Write $S_z$ for the fiber over a point $z \in \A^1$.  By flatness, $|S_{\initial}| = |S|$.  The locus of $z$ such that $X_{m,n-1}^{S_{z}}=\Gr(n-1,n)$ contains all $t \neq 0$, by the isomorphism between $S_z$ and $S$.  By Corollary~\ref{co:grclosed}, that locus must contain $0$ as well.  In other words, $X_{m,n-1}^{S_{\initial}}=\Gr(n-1,n)$.  So it suffices to prove Proposition~\ref{prop:Xm equals Gr} in the case of a Borel-fixed subscheme.

Let $S_{\initial}$ be defined by the Borel-fixed monomial ideal $J$ and let $\Lambda$ be the set of standard monomials for $J$, i.e. the monomials that do not lie in $J$.  Let $\Lambda_0\subseteq \Lambda$ be the set of standard monomials in the variables $x_2,\dots,x_n$.  Since $\Lambda_0$ is a basis for $\kk'[x_1, \ldots, x_n] / (J , x_1)$, we have that $|\Lambda_0|$ is the degree of the intersection of $S_{\initial}$ with the hyperplane $x_1 = 0$, whence $|\Lambda_0| \geq m$ by hypothesis.  In fact, though we won't need this, it is not hard to see that for a Borel-fixed $S_{\initial}$, the hyperplane $x_1$ has the minimal intersection with $S_{\initial}$ among all hyperplanes, so that $X_{m,n-1}^S=\Gr(n-1,n)$ is equivalent to $|\Lambda_0| \geq m$.

Now suppose $|\Lambda_0| \geq {b \choose n-1}$.  Then  
\beq
|\Lambda| =  |\Lambda_0| + (|\Lambda| - |\Lambda_0|) \geq {b \choose n-1} + {b \choose n} = {b+1 \choose n}
\eeq
where the inequality is Corollary~\ref{co:frontier}.  This proves Proposition~\ref{prop:Xm equals Gr} in the case $k = n-1$.

We now consider the general case.  Assume that $X_{m,k}^S=\Gr(k,n)$.  We fix some $(k+1)$-plane $V$ through the origin.  By hypothesis, $|S \cap V'|\geq m$ for every $k$-plane $V'$ through the origin; in particular, $|S \cap V'|\geq m$ for all $V'$ contained in $V$.  It follows that $X_{m,k}^{S \cap V}$ is the full Grassmannian $\Gr(k,V)$.   It follows from the $k=n-1$ case of Proposition~\ref{prop:Xm equals Gr} that $|S \cap V| \geq {b+1 \choose k+1}$.  This holds for every $(k+1)$-plane $V$ through the origin.  In particular, taking $m' = {b+1 \choose k+1}$, we have that $X_{m',k+1}^S = \Gr(k+1,n)$.  Iterating this argument yields the desired result.

\end{proof}

\section{Proof of Theorem~\ref{thm: new main}}\label{sec:proof of main theorem}
\begin{proof}[Proof of Theorem~\ref{thm: new main}]
For part (1) of the theorem, we assume that $X^S_{m,k}=\Gr(k,n)$.  We then apply Proposition~\ref{prop:Xm equals Gr} to obtain the theorem.

We now assume that $X^S_{m,k}\ne \Gr(k,n)$.  Then one of the $(|S|- m +1)\times (|S| - m +1)$-minors defining $J_{X_m}$ is nonzero, and Corollary~\ref{cor:local Xm symbolic} implies that
\[
J_{X_m} \subseteq I_{X_{2m}}^{(m+1)}.
\]

By \cite[Theorem~1.1(c)]{hochster-huneke}, if we let $\mathfrak m$ be the irrelevant ideal for the homogeneous coordinate ring $R$ of the Grassmanian, then we see that 
\[
\mathfrak m^{n+1} I_{X_{2m}}^{(m+1)} \subseteq {I_{X_{2m}}}^{\lfloor\frac{m+1}{n}\rfloor}.
\]
Since $J_{X_m}$ is generated in degree $C(|S| - m +1)$ by Lemma~\ref{lem:degree defining equations}, it follows that $\mathfrak m^{n+1}J_{X_m}$ is generated in degree $C(|S|-m+1)+n+1$.  Thus ${I_{X_{2m}}}$ must have some generators of degree at most $\frac{C(|S|-m+1)+n+1}{\lfloor\frac{m+1}{n}\rfloor}$.  

Now, if $m+1 <n$ then we can simply choose $C_2 = n$ and part (2) is trivial.  Otherwise, we can complete the proof of part (2) of the theorem by providing a constant $C_2$ depending only on $n$ and $k$ such that
\[
\frac{C(|S|-m+1)+n+1}{\lfloor\frac{m+1}{n}\rfloor}\leq C_2 \frac{|S|}{m},
\]
noting that the expression on the left is well defined because the denominator is $>0$.  
This yields part (2) of the theorem.
\end{proof}

\section{Proof of $k$-plane Furstenberg bound}\label{sec: proof of furstenberg}
\begin{proof}[Proof of Theorem~\ref{th:main}]
We first prove the theorem in the case $k=n-1$.  We apply Theorem~\ref{thm: new main}, setting $m:=\frac{q^c}{2}$.  If $X_{m,n-1}^S=\Gr(n-1,n)$ then we are done by part (1) of Theorem~\ref{thm: new main}.   Otherwise, part (2) of Theorem~\ref{thm: new main} implies that $X_{2m,n-1}^S$ lies in a hypersurface of degree at most $C_2\frac{|S|}{m}$.  However, since $X_{2m,n-1}^S$ contains all $\F_q$-rational points of $\Gr(n-1,n)$, any such hypersurface must have degree at least $q+1$. It follows that
\[
q+1 \leq C_2\frac{|S|}{m}.
\]
Since $m=\frac{q^c}{2}$ we obtain $|S|\geq C_2(q+1)(\frac{q^c}{2})\geq C_2q^{1+c}$.  Since $c\in [0,n-1]$, we have $1+c\geq cn/(n-1)$ and hence $|S|\geq C_3q^{cn/(n-1)}$ for all $q\gg 0$.

We can obtain the case of general $k$ by an iterative argument exactly parallel to the one in the proof of Proposition~\ref{prop:Xm equals Gr}.  Suppose that $S$ is a $0$-dimensional subscheme of $\A^n/\F_q$.  Without loss of generality we replace $S$ by its dilation, so we may suppose it is supported at $0$ and invariant under $\mathbb{G}_m$.

Assume that $S$ has an $m$-rich $k$-plane in every direction.  Since $S$ is supported at the origin, this is to say that $|S \cap V| \geq m$ for every $\F_q$-rational $k$-plane through the origin.  Fix some some $(k+1)$-plane $V$ through the origin.  Then $|S \cap V'|\geq m$ for all $V'$ contained in $V$.  Now the proof given above of Theorem~\ref{th:main} in the case $k=n-1$ implies that $|S \cap V| \gtrsim m^{(k+1)/k}$, and this holds for every $(k+1)$-plane $V$.  Iterating the argument for $k+2, k+3,\dots, n-1$, we get Theorem~\ref{th:main}.
\end{proof}

\begin{remark}
If you trace the constants with a bit more care, you can obtain the following more precise lower bound, at least asymptotically in $q$.  Fix any $\epsilon>0$.  Assume that $|S\cap V|\geq \frac{q^c}{k!}$ for every $k$-plane $V\in \Gr(k,n)(\F_q)$.  Then $|S|\geq (1-\epsilon) \frac{q^{cn/k}}{n!}$ for $q\gg_\epsilon 1$.  

The key point is that $\frac{q^c}{k!}\geq \binom{\lfloor q^{c/k}\rfloor}{k}$ and hence by iteratively applying Proposition~\ref{prop:Xm equals Gr}, we get that the intersection of $S$ with every hyperplane is at least $\binom{\lfloor q^{c/k}\rfloor}{n-1}$.  Let $m=\binom{\lfloor q^{c/k}\rfloor}{n-1}$.

If $X_{m,n-1}^S = \Gr(n-1,n)$ then we apply Proposition~\ref{prop:Xm equals Gr} again to obtain
\[
|S|\geq \binom{\lfloor q^{c/k}\rfloor}{n}
\]
which grows like $ \frac{q^{cn/k}}{n!}$ as $q\to \infty$, and hence is greater than $(1-\epsilon) \frac{q^{cn/k}}{n!}$ for $q \gg_\epsilon 1$.  On the other hand, for $X_{m,n-1}^S\ne \Gr(n-1,n)$, since $X_{m,n-1}^S$ contains all of the $\F_q$-points, the minimal degree of a hypersurface containing $X_{m,n-1}^S$ is at least $q$, and part (2) of Theorem~\ref{thm: new main} yields
\[
q\leq \frac{|S|-m+n+2}{\lfloor\frac{m+1}{n}\rfloor}.
\]
Using the fact that $m=\binom{\lfloor q^{c/k}\rfloor}{n-1}$ and $q^{1+c(n-1)/k}\geq q^{cn/k}$, we get the desired bound in this case as well.
\end{remark}

\section{Relation with the $k$-plane restriction conjecture}\label{sec:restriction}

One may ask how far the methods of the present paper go towards proving the $k$-plane restriction conjecture formulated in \cite{eot}, or even an extension of that conjecture to a possibly non-reduced setting as in Theorem~\ref{th:main}.  One immediate obstacle is that the most natural extension of the restriction conjecture is false, even when $k=1$, as we explain below.

The restriction conjecture concerns a certain maximal operator on real-valued functions $f$ on $\F_q^n$.  Namely:  we define a function $T_{n,k}$ on $\Gr(k,n)$ by assigning to a $k$-plane direction $\omega$ the supremum, over all $k$-planes $V$ parallel to $\omega$, of $\sum_{v \in V} |f(v)|$.  Then the restriction conjecture proposes a bound for this operator:
\begin{equation}
\label{restriction}
||T_{n,k} f||_n \ll |Gr(k,n)(\F_q)|^{1/n} ||f||_{n/k}.
\end{equation}

One way to express this conjecture more geometrically is as follows.  The bound is invariant under scaling $f$, so we can scale $f$ up until replacing $f$ with a nearby integer-valued function modifies the norm negligibly.  Then we define the scheme $S_f$ to be the union, over all $x \in \F_q^n$, of a fat point of degree $\lfloor f(x)^{1/k} \rfloor$ supported at $x$.

Thus,
\beq
|S_f| \sim \sum_x f(x)^{n/k} = ||f||_{n/k}^{n/k}
\eeq
and
\beq
|S_f \cap V| = \sum_{v \in V} f(v)
\eeq
so we can express $T_{n,k}f(\omega)$ as the supremum of $|S_f \cap V|$ over all planes $V$ parallel to $\omega$.  In other words, both sides of the conjectural inequality \eqref{restriction} are naturally expressed in terms of the geometry of the scheme $S_f$ and its restriction to $k$-planes.  For a general $0$-dimensional subscheme $S \subset \A^n$, we write $T_{n,k}(S)$ for the function on $\Gr(k,n)(\kk)$ defined by
\beq
T_{n,k}(S)(\omega) = \sup_{V || \omega} |S \cap V|.
\eeq
Then we can ask whether we have an inequality
\begin{equation}
\label{grestriction}
||T_{n,k}(S)||_n \lesssim |Gr(k,n)(\F_q)|^{1/n} |S|^{k/n}.
\end{equation}
for all $0$-dimensional $S$; the case $S = S_f$ is more or less equivalent to the $k$-plane restriction conjecture in \cite{eot}.

Unfortunately, \eqref{grestriction} does not hold for all $S$.  For example, take $k=1,n=2$, and let $S$ be the scheme $\Spec \F_q[x,y]/(x,y^N)$.  That is, $S$ is a scheme of degree $N$, supported at the origin, which is contained in the line $x=0$.  Then $T_{2,1}(S)$ is $N$ in the vertical direction and $1$ in all other directions; so $||T_{n,k}(S)||_2$ is $(N^2+q)^{1/2}$, while $|S|^{k/n} = N^{1/2}$.  Then the desired inequality \eqref{grestriction} becomes
\beq
(N^2 + q)^{1/2} \lesssim (q+1)^{1/2} N^{1/2}
\eeq
which holds only when $N \ll q$  

This is in some sense the same issue that arises in Remark~\ref{re:0k}, where our theorem of Furstenberg schemes requires a condition $c \in [0,k]$ which is automatically satisfied for Furstenberg sets.  Something similar appears to be necessary to formulate the correct restriction conjecture for schemes.  For example:  if $S$ is actually of the form $S_f$ and is contained in the line $x=0$, it must be reduced, from which it follows that $|S| < q$.  It is an interesting question whether one can prove \eqref{grestriction} under some geometric conditions on $S$.  Ideally, these conditions would be lenient enough to include the schemes $S_f$ for all real-valued functions $f$.  One natural such question is as follows.

\begin{question} Suppose $S$ is a $0$-dimensional subscheme of $\A^n/\F_q$ which is contained in a complete intersection of $n$ hypersurfaces of degree $Q$.  What upper bounds on the schemes $X_{m,k}^S$ -- say, on their Hilbert functions -- can we obtain in terms of $|S|$ and $Q$? 
\label{q:ci}
\end{question}

Information about Question~\ref{q:ci} would give insight into the case where $f$ was an indicator function of a set $S$, since in that case $S$ is contained in $\A^n(\F_q)$, which is a complete intersection of the hypersurfaces $x_i^q - x_i$ as $i$ ranges from $1$ to $n$.

\section{Examples}\label{sec:examples}
\begin{example}
If $|S| \leq q^{c+\alpha}$ and $c+\alpha\leq cn/k$, then Theorem~\ref{thm: new main} implies that all of the $q^c$-rich $k$-planes of $S$ must lie on a hypersurface of degree $\leq q^{\alpha}$.  For instance, if $|S| \approx q^c$ then all of the $q^c$-rich $k$-planes of $S$ must lie on a hypersurface of bounded degree.
\end{example}

\begin{example}
Let $k=2$ and $n=4$, and let $I$ be the monomial ideal whose quotient ring has basis $\{1,x_1,x_2,x_3,x_4,x_4^2\}$.  Note that $|S|=6$.

The source of $\overline{\Phi}$ is a nontrivial vector bundle, and hence we cannot simply write the map as a simple matrix. We thus consider the open subset of $\Gr(2,4)$ where the Pl\"ucker coordinate $p_{12}$ is nonzero, and here we can write any $2$-plane uniquely as the vanishing set:
\[
\begin{cases}
x_1 + \frac{p_{23}}{p_{12}}x_3 + \frac{p_{24}}{p_{12}}x_4&=0\\
x_2 + \frac{p_{13}}{p_{12}}x_3 + \frac{p_{14}}{p_{12}}x_4&=0\\
\end{cases}
\]
Over this open subset, the map $\overline{\Phi}$ can be written as a matrix
\[
\Phi = 
\bordermatrix{&1&x_1&x_2&x_3&x_4&x_4^2&1&x_1&x_2&x_3&x_4&x_4^2\cr
            1 &0&0&0&0&0&0&0&0&0&0&0&0\cr
             x_1&1&0&0&0&0&0&0&0&0&0&0&0\cr
             x_2&0&0&0&0&0&0&1&0&0&0&0&0\cr
             x_3&\frac{p_{23}}{p_{12}}&0&0&0&0&0&\frac{p_{13}}{p_{12}}&0&0&0&0&0\cr
             x_4&\frac{p_{24}}{p_{12}}&0&0&0&0&0&\frac{p_{14}}{p_{12}}&0&0&0&0&0\cr
             x_4^2&0&0&0&0&\frac{p_{24}}{p_{12}}&0&0&0&0&0&\frac{p_{14}}{p_{12}}&0\cr
             }
\]

Recall that we compute $X_{m,2}^S$ by the $(|S|-m+1)$-minors of $\overline{\Phi}$.  If $m = 3$ then we get $4\times 4$-minors of $\overline{\Phi}$ which are all $0$, and hence $X_{3,2}^S$ contains every point in the open subset $p_{12}\ne 0$ and thus  $X_{3,2}^S=\Gr(2,4)$.  If $m\geq 5$, then $X_{m,2}^S\cap \{p_{12}\ne 0\}=\emptyset$ since the rank of $\overline{\Phi}$ is $2$.  The case $m=4$ is the most interesting, as then $X_{4,2}^S\cap \{p_{12}\ne 0\}$ is defined by the ideal of $3\times 3$ minors of $\overline{\Phi}$.  This yields the ideal
\[
J = \left\langle \frac{p_{24}}{p_{12}},  \frac{p_{14}}{p_{12}}\right\rangle.
\]
Thus, $\Sigma_{4,2}^S\cap \{p_{12}\ne 0\}$ is the set of all $2$-planes of the form
\[
\begin{cases}
x_1 + \frac{p_{23}}{p_{12}}x_3 &=0\\
x_2 + \frac{p_{13}}{p_{12}}x_3 &=0.\\
\end{cases}
\]
\end{example}


\begin{bibdiv}
\begin{biblist}

\bib{bct}{article}{
   author={Bennett, Jonathan},
   author={Carbery, Anthony},
   author={Tao, Terence},
   title={On the multilinear restriction and Kakeya conjectures},
   journal={Acta Math.},
   volume={196},
   date={2006},
   number={2},
   pages={261--302},
}

\bib{bbc}{article}{
	author={Bennett, Jonathan},
	author={Bez,Neal},
	author={Carbery,Anthony},
	title={Heat-flow monotonicity related to the Hausdorff - Young inequality},
	journal={Bulletin of the London Mathematical Society},
	pages={bdp073},
	date={2009},
}
	
	
\bib{bruns-herzog-book}{book}{
   author={Bruns, Winfried},
   author={Herzog, J{\"u}rgen},
   title={Cohen-Macaulay rings},
   series={Cambridge Studies in Advanced Mathematics},
   volume={39},
   publisher={Cambridge University Press, Cambridge},
   date={1993},
   pages={xii+403},
}

\bib{dvir:kakeya}{article}{
   author={Dvir, Zeev},
   title={On the size of Kakeya sets in finite fields},
   journal={J. Amer. Math. Soc.},
   volume={22},
   date={2009},
   number={4},
   pages={1093--1097},
}

\bib{ein-lazarsfeld-smith}{article}{
   author={Ein, Lawrence},
   author={Lazarsfeld, Robert},
   author={Smith, Karen E.},
   title={Uniform bounds and symbolic powers on smooth varieties},
   journal={Invent. Math.},
   volume={144},
   date={2001},
   number={2},
   pages={241--252},
}

\bib{eisenbud-book}{book}{
   author={Eisenbud, David},
   title={Commutative algebra with a view toward algebraic geometry},
   series={Graduate Texts in Mathematics},
   volume={150},
   publisher={Springer-Verlag, New York},
   date={1995},
   pages={xvi+785},
}

\bib{eot}{article}{
   author={Ellenberg, Jordan S.},
   author={Oberlin, Richard},
   author={Tao, Terence},
   title={The Kakeya set and maximal conjectures for algebraic varieties
   over finite fields},
   journal={Mathematika},
   volume={56},
   date={2010},
   number={1},
   pages={1--25},
}

\bib{haiman-sturmfels}{article}{
   author={Haiman, Mark},
   author={Sturmfels, Bernd},
   title={Multigraded Hilbert schemes},
   journal={J. Algebraic Geom.},
   volume={13},
   date={2004},
   number={4},
   pages={725--769},
}

\bib{hartshorne}{book}{
   author={Hartshorne, Robin},
   title={Algebraic geometry},
   note={Graduate Texts in Mathematics, No. 52},
   publisher={Springer-Verlag, New York-Heidelberg},
   date={1977},
   pages={xvi+496},
}

\bib{hochster-huneke}{article}{
   author={Hochster, Melvin},
   author={Huneke, Craig},
   title={Comparison of symbolic and ordinary powers of ideals},
   journal={Invent. Math.},
   volume={147},
   date={2002},
   number={2},
   pages={349--369},
}
	

\bib{wolff99}{article}{
   author={Wolff, Thomas},
   title={Recent work connected with the Kakeya problem},
   conference={
      title={Prospects in mathematics},
      address={Princeton, NJ},
      date={1996},
   },
   book={
      publisher={Amer. Math. Soc., Providence, RI},
   },
   date={1999},
   pages={129--162},
}

\bib{zhang}{article}{
author = {Zhang, Ruixiang},
title = {On configurations where the Loomis-Whitney inequality is nearly sharp and applications to the Furstenberg set problem},
journal={Mathematika},
volume={61},
number={1},
date={2015},
pages={145--161},
}



\end{biblist}
\end{bibdiv}
\end{document}